\documentclass[11pt, oneside]{article}   	
\usepackage{geometry}                		
\usepackage{graphicx}				

\usepackage{indentfirst}

\usepackage{ulem}	
\usepackage{amsthm}
\usepackage{amsfonts}
\usepackage{amsmath}
\usepackage{amscd}
\usepackage{graphicx}
\usepackage{enumerate}
\usepackage{amssymb}	
\usepackage{verbatim}
\usepackage{cancel}
\usepackage{cite}
\usepackage{textcomp}								
\usepackage{color}

\newtheorem{theorem}{Theorem}[section]
\newtheorem{definition}{Definition}[section]
\newtheorem{proposition}{Proposition}[section]
\newtheorem{lemma}{Lemma}[section]
\newtheorem{remark}{Remark}[section]

\newtheorem{corollary}{Corollary}[section]

\title{\textsc{On density of infinite subsets II: dynamics on homogeneous spaces}}
\author{\textsc{Changguang Dong}}
\date{}

\begin{document}
\maketitle
\begin{abstract}
Let $G$ be a noncompact semisimple Lie group, $\Gamma$ be an irreducible cocompact lattice in $G$, and $P<G$ be a minimal parabolic subgroup. We consider the dynamics of $P$ acting on $G/\Gamma$ by left translation. For any infinite subset $A\subset G/\Gamma$, we show that, for any $\epsilon>0$, there is a $g\in P$ such that $gA$ is $\epsilon$-dense. We also prove a similar result for certain discrete group actions on $\mathbb T^n$.
\end{abstract}

\section{Introduction and Results}

In this note, we make further progress on density of infinite subset initiated in \cite{do1}. We will in particular focus on the D.I. problem.

To be more precise, let $Y$ be a compact metric space, and $G$ be a locally compact second countable topological (semi-)group which acts on $Y$ by homeomorphisms. Let $A$ be an infinite subset of $Y$, we can consider the set containing all subsets of the form $gA:=\{\alpha(g)x|x\in A\}$ for a $g\in G$. For the fixed $A$, we would like to know: for any $\epsilon>0$, whether there exists a $g\in G$ such that $gA$ is $\epsilon$-dense in $Y$, or equivalently $d^H_Y(gA,Y)<\epsilon$. We call this dense iteration problem simply {\bf D.I. problem}. 

Here is a nontrivial result in this direction. Let $S^1=\mathbb R/\mathbb Z$ be the standard circle, and $T_\alpha:S^1\to S^1$ be the translation map: $x\mapsto x+\alpha\;(mod\, 1)$. A theorem of Glasner \cite{gl} asserts that if $X$ is an infinite subset of $S^1$, then for any $\epsilon>0$, there exists an integer $n$ such that the dilation $nX := \{nx\; (mod\, 1) : x\in X\}$ is $\epsilon$-dense. This gives an affirmative answer to the D.I. problem in the case of the natural action by multiplication of $\mathbb N$ on the circle $S^1$. 

In view of this result, we made the following definitions in \cite{do1}.

\begin{definition}
Given a $G$ action on a metric space $Y$, if an infinite subset $A$ satisfies that for any $\epsilon>0$, there exists a $g\in G$ such that $gA$ is $\epsilon$-dense in $Y$, then $A$ is called {\bf Glasner set} with respect to $(Y,d,G)$.
\end{definition}

\begin{definition}
Given a $G$ action on a metric space $Y$, if any infinite subset $A$ is a Glasner set, then we say the dynamical system $(Y,d,G)$ has {\bf Glasner property}.
\end{definition}

Using our definition, the system $(S^1,d_L,\mathbb N)$ has the Glasner property. We also proved in \cite{do1} that for any positive integer $N\ge 2$, the system $(\mathbb T^N,d_L,SL(N, \mathbb Z))$ has Glasner property. 

In this note, we consider ``large" group acting on homogeneous spaces. Recall that, a subgroup $F$ of a real algebraic group $G$ is called {\bf epimorphic} in $G$ if any $F$-fixed vector is also $G$-fixed for any finite dimensional algebraic linear representation of $G$. As an example, the parabolic group of a semisimple real Lie group without compact factor is epimorphic. Our first result is

\begin{theorem}\label{t1}
Let $G$ be a connected semisimple real Lie group with trivial center and no compact factor, $\Gamma$ be an irreducible cocompact lattice in $G$, and $P<G$ be an epimorphic subgroup. Consider $P$ acting on $G/\Gamma$ by left translation. Then $(G/\Gamma, d, P)$ has Glasner property.
\end{theorem}

Here, a lattice $\Gamma$ in a connected semisimple Lie group $G$ with finite center is {\it irreducible} if the projection of $\Gamma$ to $G/H$ is dense for every nontrivial connected normal subgroup $H\le G$.

Our second result is a generalization of \cite[Theorem 1.1]{do1}.

\begin{theorem}\label{t2}
Let $n\ge 2$, and $\Gamma$ be a subgroup of $GL(n,\mathbb Z)$. Assume the Zariski closure of $\Gamma$ is semisimple, Zariski connected and with no compact factor, and acts irreducibly on $\mathbb Q^n$. Then the system $(\mathbb T^n,d_L,\Gamma)$ has Glasner property.
\end{theorem}

\begin{remark}
A particular case of Theorem \ref{t2} is when $\Gamma<SL(n,\mathbb Z)$ is Zariski dense in $SL(n,\mathbb R)$.
\end{remark}

\begin{remark}
We think, a version of Theorem \ref{t2} can be obtained similarly for certain groups acting on nilmanifold.
\end{remark}

Our main ingredient is the classification of orbit closure of certain group action. We heavily use the orbit closure results in \cite{bq},\cite{sw}.

\section{Facts from homogeneous dynamics}

\subsection{Orbit closure}

The action of epimorphic subgroups on homogeneous spaces is well understood either in the case of invariant measure classification \cite{mo} or in the case of orbit closure \cite{sw}. Here we will use the result on orbit closure.

\begin{theorem}[Corollary 1.3,\cite{sw}]\label{sw1}
Let $F < G < L$ be an inclusion of connected real algebraic groups such that $F$ is epimorphic in $G$. Then any closed $F$-invariant subset in $L/\Lambda$ is $G$-invariant, where $\Lambda$ is a lattice in $L$.
\end{theorem}

Hence we have the following

\begin{corollary}\label{co1}
Let $G, \Gamma,P$ be given as in Theorem \ref{t1}. For any integer $k$,  consider the $P$ (or $G$) action on $(G/\Gamma)^k$ defined by $g(x_1,\ldots,x_k)=(gx_1,\ldots,gx_k)$ for $g\in P$ (or $G$) and $(x_1,\ldots,x_k)\in (G/\Gamma)^k$. Then for any $\bar x\in (G/\Gamma)^k$, the closure of $P$ orbit of $\bar x$ coincides with the closure of $G$ orbit of $\bar x$.
\end{corollary}
\begin{proof}
Apply Theorem \ref{sw1} with $L=G^k$, $F=P$ and $\Lambda=\Gamma^k$, the result follows.
\end{proof}

\subsection{Commensurability group of $\Gamma$}

Let $\gamma\in G$, $\gamma$ is an element of the {\it commensurator} of $\Gamma$ in $G$ if $\Gamma\cap\gamma\Gamma\gamma^{-1}$ has finite index in both $\Gamma$ and $\gamma\Gamma\gamma^{-1}$. We write
$Comm(\Gamma)$ for the commensurator of $\Gamma$ in $G$, namely, $Comm(\Gamma)=\{\gamma\in G:[\Gamma:\Gamma\cap\gamma\Gamma\gamma^{-1}]<\infty,[\gamma\Gamma\gamma^{-1}:\Gamma\cap\gamma\Gamma\gamma^{-1}]<\infty\}$. It is known that $Comm(\Gamma)$ is a subgroup of $G$. Moreover, $Comm(\Gamma)$ satisfies a dichotomy (see \cite{zi}): either $Comm(\Gamma)$ contains $\Gamma$ as a subgroup of finite index, or $Comm(\Gamma)$ is dense in $G$. In fact, it is a celebrated theorem of Margulis that this is precisely the dichotomy of arithmeticity v.s. non-arithmeticity.

\begin{theorem}[Margulis, \cite{zi},\cite{ma}]\label{ma1}
Let $G$ be a connected semisimple real Lie group with trivial center and no compact factor, $\Gamma<G$ be an irreducible cocompact lattice. Then either $\Gamma$ is arithmetic and $Comm(\Gamma)$ is dense in $G$ (w.r.t. Hausdorff topology), or $\Gamma$ is not arithmetic and $\Gamma$ is a finite index subgroup of $Comm(\Gamma)$.
\end{theorem}

The commensurators of $\Gamma$ play an important role in analyzing the dynamics on $G/\Gamma$. In fact, as we will describe later, they will give nontrivial self joinings of the $G$ action on $G/\Gamma$.

\subsection{Benoist-Quint Theorems}

We are going to use several results from \cite{bq}. In order to be self contained, we collect in the following those which will be used in the proofs.

\begin{theorem}[Benoist-Quint,\cite{bq}]\label{bq1}
Let $G$ be a connected semisimple real Lie group with trivial center and no compact factor, $\Gamma<G$ be an irreducible cocompact lattice. Let $\Lambda<G$ be a Zariski dense subgroup. Consider $\Lambda$ acting on $G/\Gamma$ by left translations, then
\begin{itemize}
\item[(1)] every $\Lambda$ orbit closure is either discrete (and hence finite) or $G/\Gamma$. In particular, this is true for the action of any finite index subgroup of $\Gamma$,
\item[(2)] any sequence of distinct finite $\Lambda$ orbits has $G/\Gamma$ as the only limit in the Hausdorff topology. 
\end{itemize}
\end{theorem}

\begin{theorem}[Benoist-Quint,\cite{bq}]\label{bq2-0}
Let $n\ge 2$, and $\Gamma$ be a subgroup of $GL(n,\mathbb Z)$. Assume the Zariski closure of $\Gamma$ is semisimple, Zariski connected and with no compact factor. Consider $\Gamma$ acting on $\mathbb T^n$ naturally by automorphisms, then every $\Gamma$-orbit closure is a finite homogeneous union of affine submanifolds.
\end{theorem}
\begin{remark}
These affine submanifolds are defined over $\mathbb Q$, by which we mean they are given by some affine equations with coefficients in $\mathbb Q$.
\end{remark}

\begin{theorem}[Benoist-Quint,\cite{bq}]\label{bq2}
Let $n\ge 2$, and $\Gamma$ be a subgroup of $GL(n,\mathbb Z)$. Assume the Zariski closure of $\Gamma$ is semisimple, Zariski connected and with no compact factor, and acts irreducibly on $\mathbb Q^n$. Consider $\Gamma$ acting on $\mathbb T^n$ naturally by automorphisms, then
\begin{itemize}
\item[(1)] every $\Gamma$ orbit closure is either discrete (and hence finite) or $\mathbb T^n$. In particular, this is also true for the action of any finite index subgroup of $\Gamma$,
\item[(2)] any sequence of distinct finite $\Gamma$ orbits has $\mathbb T^n$ as the only limit in the Hausdorff topology. 
\end{itemize}
\end{theorem}
\begin{remark}
The above theorem applies when $\Gamma<SL(n,\mathbb Z)$ is Zariski dense in $SL(n,\mathbb R)$.
\end{remark}

\section{Orbit closure of $G$ action on products of $(G/\Gamma,Haar)$}

Let $L$ be a group. Consider two measure preserving systems $(L,X_1,\mu)$ and $(L,X_2,\nu)$, a joining is a measure on $X_1\times X_2$ which is invariant under the $L$ action, and coincides with $\mu$ (respectively $\nu$) when projects to $X_1$ (respectively $X_2$). A self joining of$(L,X,\mu)$ is a joining for $(L,X,\mu)$ and $(L,X,\mu)$. In this subsection, we describe all ergodic self joinings of $G$ action on $(G/\Gamma,Haar)$.

As $G$ is generated by unipotent elements, applying Ratner rigidity Theorems, any ergodic self joining either coincides with the product Haar measure, or it reduces to a Haar measure supported on a closed $G$ invariant homogeneous submanifold. The latter is related to the elements in $Comm(\Gamma)$, and is essentially a finite extension of Haar measrure on $G/\Gamma$. There are many ways to describe such self joinings. We present a description via $G$ equivariant maps.

For any $\gamma\in Comm(\Gamma)$, let $\hat\Gamma=\Gamma\cap\gamma\Gamma\gamma^{-1}$, we have a series of $G$ equivariant maps: $$G/\hat\Gamma\hookrightarrow G/\hat\Gamma\times G/\hat\Gamma\to G/\Gamma\times G/(\gamma\Gamma\gamma^{-1})\to G/\Gamma\times G/\Gamma$$defined by $$(x\hat\Gamma)\mapsto(x\hat\Gamma,x\hat\Gamma)\mapsto (x\Gamma, x\gamma\Gamma\gamma^{-1})\mapsto (x\Gamma,x\gamma\Gamma).$$ Then the Haar measure on $G/\hat\Gamma$ will be mapped to a $G$ invariant measure on $G/\Gamma\times G/\Gamma$. We will call this self joining supported on a graph. 

\begin{lemma}\label{le1}
For any $\gamma\in Comm(\Gamma)$, the $\Gamma$ orbit of point $\gamma\Gamma$ in $G/\Gamma$ contains finite many points. On the other hand, if $\Gamma$ orbit of a point $x \in G/\Gamma$ contains finite many points, then $x=\gamma\Gamma$ for some $\gamma\in Comm(\Gamma)$.
\end{lemma}
\begin{proof}
For $\gamma\in Comm(\Gamma)$, let $\hat\Gamma=\Gamma\cap\gamma\Gamma\gamma^{-1}$, then $\hat\Gamma$ is the stabilizer of $\gamma\Gamma$. Combine this with the fact that $[\Gamma: \hat\Gamma]<\infty$, we obtain the first claim. The second claim follows similarly by considering the stabilizer.
\end{proof}

\begin{proposition}\label{pro1}
Combining with the product Haar measure, these exhaust all ergodic self joinings on $G/\Gamma\times G/\Gamma$.
\end{proposition}

\begin{proof}
Let $\mu$ be an ergodic self joining on $G/\Gamma\times G/\Gamma$, and assume that $\mu\neq Haar\times Haar$. By Theorem \ref{bq1}, $\mu$ is a Haar measure supported on a $G$-invariant homogeneous space. Let $W$ be the support of $\mu$. Then $W\cap (\{\Gamma\}\times G/\Gamma)$ is finite. Indeed, notice that the $G$ action on $\{\Gamma\}\times G/\Gamma$ reduces to a $\Gamma$ action on $G/\Gamma$, then if $W\cap (\{\Gamma\}\times G/\Gamma)$ is not finite, by Theorem \ref{bq1}, the $\Gamma$ orbit must be dense, this contradicts to the finiteness of $\mu$ and $\mu\neq Haar\times Haar$. 

Now by Lemma \ref{le1}, there is a $\gamma\in Comm(\Gamma)$ such that $W\cap (\{\Gamma\}\times G/\Gamma)=\Gamma\circ(\Gamma,\gamma\Gamma)$. In particular, $(\Gamma,\gamma\Gamma)\in W$. From here, it is easy to see that the measure $\mu$ is supported on a graph just as what we described before.
\end{proof}
By Proposition \ref{pro1}, we have
\begin{corollary}\label{co2}
The orbit closure of any point will be given by the support of some ergodic self joining.
\end{corollary}

Let $(x\Gamma),(y\Gamma)$ be two points on $G/\Gamma$. Define a relation $\sim$: $(x\Gamma)\sim (y\Gamma)$ if there exists a $\gamma\in Comm(\Gamma)$ such that $x= y\gamma$. It is straightforward to see that $\sim$ is an equivalence relation.

\begin{theorem}\label{pj-1}
Let $(a_1,\ldots,a_\ell)\in (G/\Gamma)^{\ell}$, $\ell\ge 1$. If there is {\bf no} pair $i,j$ with $i\neq j$ such that $a_i\sim a_j$, then the $G$-orbit closure of $(a_1,\ldots,a_\ell)\in (G/\Gamma)^{\ell}$ is $(G/\Gamma)^{\ell}$.
\end{theorem}
\begin{proof}
By induction on $\ell$. When $\ell=1$, it is true because $G$ action is minimal. When $\ell=2$, this is a corollary of Proposition \ref{pro1}. 

Now assume it is true for $\ell=1,2,\ldots,k$, we want to prove the case that $\ell=k+1$. Since the theorem is true for $\ell=k$, apply Ratner's results on measure rigidity and orbit closure, the $G$-orbit closure of $(a_1,\ldots,a_{k+1})$ is algebraic. Let $H$ be an algebraic group such that $$W:=\overline{G.(a_1,\ldots,a_{k+1})}=H.(a_1,\ldots,a_{k+1}).$$Then it follows that $G^k\subset H\subset G^{k+1}$ and $vol(H/(H\cap \Gamma^{k+1}))<\infty$. If $H=G^{k+1}$, then we are done. 

Now if $H\subsetneq G^{k+1}$, then $H=G^k$. Let $\pi_k:(G/\Gamma)^{k+1}\to (G/\Gamma)^{k}$ be the projection map to the first $k$ coordinates. By assumption, $\pi_k(W)=(G/\Gamma)^{k}$. Then by algebraicity of $W$, $$\#(W\cap \pi^{-1}_k(\bar x))<\infty\;\text{for any }\bar x\in (G/\Gamma)^{k}.$$This enables us to take finite extension of $(G/\Gamma)^{k+1}$ to obtain $(G/\Gamma')^{k+1}$, such that the orbit closure of $(a_1,\ldots,a_{k+1})$ intersects the fibre built by the corresponding projection map $\pi_k'$ with exactly one point. Let $W'$ be the orbit closure. It is given by $(x,\omega (x))\in(G/\Gamma')^k\times G/\Gamma'$ for some $G$ equivariant map $\omega:(G/\Gamma')^k\to G/\Gamma'$. In fact, $\omega$ comes from a group homomorphism from $G^k$ to $G$ such that $\omega(\Gamma'^k)=\Gamma'$. From here, one have that $\omega$ maps one coordinate of $(G/\Gamma')^k$ to its image. Let it be the $i$th coordinate. Then combine $i$th and $(k+1)$th coordinate of $(G/\Gamma')^{k+1}$, the corresponding $G$ orbit is supported on a graph in $(G/\Gamma)^2$. Therefore by Proposition \ref{pro1}, we have $a_i\sim a_{k+1}$, a contradiction to our assumption. This finishes the proof.
\end{proof}

\section{Orbit closure of certain group actions on products of $\mathbb T^n$}

The space of self joinings of discrete group actions on $\mathbb T^n$ is a little bit complicated than that of the $G$ action described in previous subsection. One reason is that there are infinitely many finite orbits on $\mathbb T^n$.

\begin{lemma}\label{to1}
Let $n\ge 2$, and $\Gamma$ be a subgroup of $GL(n,\mathbb Z)$. Assume the Zariski closure of $\Gamma$ is semisimple, Zariski connected and with no compact factor, and acts irreducibly on $\mathbb Q^n$. Let $C(\Gamma)=\{\lambda\in M(n\times n,\mathbb Z):\det\lambda\neq 0,\lambda\circ \gamma=\gamma\circ\lambda,\;\forall \gamma\in \Gamma\}$ be the space of centralizers of $\Gamma$. Then $C(\Gamma)=\{k I_n:k\neq 0, k\in\mathbb Z\}$.
\end{lemma}
\begin{proof}
Assume $\eta\in C(\Gamma)$. Let $H$ be the Zariski closure of $\Gamma$. Then by assumptions, $H$ is a semisimple group in $GL(n,\mathbb R)$ and $\eta\circ h=h\circ \eta$ for any $h\in H$. Since they are matrix Lie groups, then after conjugation simultaneously, $\eta$ is a diagonal block matrix of the diagonal form as $H$. For each simple block matrix, the corresponding $\eta$ must be a constant multiple of Identity. By the irreducibility on $\mathbb Q^n$, the multiplying constants for different blocks should be equal. Therefore $\eta$ is a constant multiple of $I_n$. Since $\eta\in M(n\times n,\mathbb Z)$, $\eta=k I_n$ for some nonzero $ k\in\mathbb Z$.
\end{proof}

We first consider orbit closures on product spaces. For any $r\ge 1$, we say $x_1,\cdots,x_r$ are {\it rationally dependent}, if there exists $a_1,\cdots,a_r\in\mathbb Z$ such that $\sum_{i=1}^ra_ix_i\in\mathbb Q^n/\mathbb Z^n$; otherwise, $x_1,\cdots,x_r$ are {\it rationally independent}.

\begin{theorem}\label{pt1}
Let $n,\Gamma,C(\Gamma)$ be as in Lemma \ref{to1}. Consider $\Gamma$ acting on $\mathbb T^n$ naturally by automorphisms. Let $x, y$ be any two points in $\mathbb T^n$, then exactly one of the following holds:
\begin{itemize}
\item[(1)] $x\in\mathbb Q^n/\mathbb Z^n$, and $y\in\mathbb Q^n/\mathbb Z^n$. The $\Gamma$ orbit closure of $(x,y)$ is discrete and hence finite;
\item[(2)] only one of $x,y$ is in $\mathbb Q^n/\mathbb Z^n$. The $\Gamma$ orbit closure of $(x,y)$ is a direct product of a finite orbit with $\mathbb T^n$;
\item[(3)] $x,y$ are rationally dependent. The $\Gamma$ orbit closure of $(x,y)$ is a finite union of rational translations of $(\phi_{\lambda,\theta})(\mathbb T^n)$ for some $\lambda,\theta\in C(\Gamma)$, where $\phi_{\lambda,\theta}:\mathbb T^n\to \mathbb T^n\times \mathbb T^n$ is defined by $\phi_{\lambda,\theta}(x)=(\lambda x,\theta x)$;
\item[(4)] $x,y$ are rationally independent. The $\Gamma$ orbit closure of $(x,y)$ is $\mathbb T^n\times\mathbb T^n$.
\end{itemize}
\end{theorem}
\begin{proof}
By Theorem \ref{bq2-0}, it is known that the $\Gamma$ orbit closure of $(x,y)$ is a finite union of affine manifold. By replacing $\Gamma$ by its finite index subgroup $\Gamma'$, we have that the $\Gamma'$ orbit closure of $(x,y)$ is an affine manifold. The cases (1) and (2) are straightforward. 

Now we turn to (3) first. When $x,y$ are rationally dependent, then there is a $z\in\mathbb T^n$ such that $x=az$ and $y=bz+q_1$ where $a,b\in \mathbb Z$ with $(a,b)=1$ and $q_1\in\mathbb Q^n/\mathbb Z^n$. Then the $\Gamma$ orbit closure of $(x,y)$ reduces to $\Gamma'$ orbit closure of $(az,bz)$. As $z$ is not a rational point, it is easy to see that the latter is $(\phi_{\lambda,\theta})(\mathbb T^n)$ with $\lambda=aI_n$ and $\theta=bI_n$.

When $x,y$ are rationally independent, the orbit closure is the product space, since there is no $\Gamma$-invariant affine submanifold containing $(x,y)$. This yields (4).
\end{proof}

\begin{corollary}
Let $n,C(\Gamma)$ be as in Lemma \ref{to1}. Consider $\Gamma$ acting on $\mathbb T^n$ naturally by automorphisms with the Lebesgue measure $m$, then there are 2 types of ergodic self joinings of this action:
\begin{itemize}
\item[(1)] $m\times m$, the product of Lebesgue measures on $\mathbb T^n\times \mathbb T^n$;
\item[(2)] average of finitely many translations of $(\phi_{\lambda,\theta})_*(m)$, where $\lambda,\theta\in C(GL(n,\mathbb Z))$, $\phi_{\lambda,\theta}:\mathbb T^n\to \mathbb T^n\times \mathbb T^n$ is defined by $\phi_{\lambda,\theta}(x)=(\lambda x,\theta x)$.
\end{itemize}
\end{corollary}

\begin{theorem}\label{pt2}
Let $n,\Gamma,C(\Gamma)$ be as in Lemma \ref{to1}. Consider $\Gamma$ acting on $\mathbb T^n$ naturally by automorphisms. For any $k$, if $x_1,x_2,\ldots,x_k$ are rationally independent, then the orbit closure of $(x_1,x_2,\ldots,x_k)$ is $(\mathbb T^n)^k$.
\end{theorem}
\begin{proof}
It follows from Theorem \ref{bq2} and the fact that there is no invariant affine submanifold containing the point $(x_1,x_2,\ldots,x_k)$, when $x_1,x_2,\ldots,x_k$ are rationally independent.
\end{proof}

\section{Proof of Theorem \ref{t1}}

Let $A$ be an arbitrary infinite subset.

Let $(x\Gamma),(y\Gamma)$ be two points on $G/\Gamma$. Consider the equivalence relation $\sim$: $(x\Gamma)\sim (y\Gamma)$ if there exists a $\gamma\in Comm(\Gamma)$ such that $(x\Gamma)= (y\gamma\Gamma)$. Notice that by Theorem \ref{pj-1}, only if $(x\Gamma) \sim (y\Gamma)$, the orbit closure of $(x\Gamma,y\Gamma)$ under $G$ will be a graph as described before. Now, we can partite $A$ into subsets $\{A_1,A_2,\ldots,A_i,\ldots\}_{i\in I}$, such that each $A_i$ contains points in one equivalence class. 

If $Card(I)=\infty$, then we can get an infinite subset $\hat A\subset A$ by simply choosing one point from each subset, say choose $a_i\in A_i$. For any $\ell>0$, the orbit closure of $(a_1,\ldots,a_\ell)\in (G/\Gamma)^{\ell}$ is $(G/\Gamma)^{\ell}$. Now for any $\epsilon>0$, let $\ell$ be great enough, then there exists $g\in G$ such that the subset $g\{a_1,\ldots,a_\ell\}=\{ga_1,\ldots,ga_\ell\}$ is $\epsilon$-dense. Therefore the set $g(A)$ is also $\epsilon$-dense. We are done in this case. Let's remark that if $\Gamma$ is not arithmetic, then $Card(I)=\infty$.

If $Card(I)<\infty$, since $A$ is an infinite subset, there exists $i\in I$ such that $A_i$ also contains infinite many points. Thus without loss of generality, afterwards assume $A$ contains points in one equivalence class. As $G$ acts transitively on $G/\Gamma$, assume that $A=\{(\Gamma),(\gamma_1\Gamma),\ldots\ldots\}$ where $\gamma_i\in Comm(\Gamma)$, and the point $(\Gamma)\in G/\Gamma$ is the only accumulating point of $A$.

\begin{lemma}\label{k1}
For any $\ell>0$, any open subset $U\subset (G/\Gamma)^\ell$, there exist a $g\in G$ and $(b_1\Gamma,\ldots,b_\ell\Gamma)\in (G/\Gamma)^\ell$ with $(b_i\Gamma)\in A$, such that $g(b_1\Gamma,\ldots,b_\ell\Gamma)\in U$.
\end{lemma}
\begin{proof}
It suffices to prove the case when $U=U_1\times \cdots\times U_\ell$, where $U_i\subset G/\Gamma$ is an open subset. We prove this by induction on $\ell$.

When $\ell=1$, since $G$ action on $G/\Gamma$ is minimal, any point in $A$ works.

Assume that when $\ell=k-1\ge 1$, the lemma is true. Now we prove it for $\ell=k$. Let $U=U_1\times \cdots\times U_k$ be an arbitrary open set. Apply the case $\ell=k-1$ for the first $k-1$ product $U_1\times \cdots\times U_{k-1}$, we thus obtain $(b_1\Gamma,\ldots,b_{k-1}\Gamma)\in (G/\Gamma)^{k-1}$ with $(b_i\Gamma)\in A$, and $g_0(b_1\Gamma,\ldots,b_{k-1}\Gamma)\in U_1\times \cdots\times U_{k-1}$ for some $g_0\in G$. Notice that the $G$ orbit closure of $(b_1\Gamma,\ldots,b_{k-1}\Gamma)$ is essentially a homogeneous $G$-space, then the stabilizer of the point $g_0(b_1\Gamma,\ldots,b_{k-1}\Gamma)$ is a discrete group $g_0\Gamma_{k-1}g_0^{-1}$, where $\Gamma_{k-1}$ is a finite index subgroup of $\Gamma$. Here $g_0\Gamma_{k-1}g_0^{-1}$ is still a cocompact lattice.

Let $\mathcal A_k:=A\backslash\{b_i\Gamma:1\le i\le k-1\}$, then $Card(\mathcal A_k)=\infty$. By Theorem \ref{bq1}, it follows that there is a $(b_k\Gamma)\in\mathcal A_k$ such that $$\left(g_0\Gamma_{k-1}g_0^{-1}(g_0b_k\Gamma)\right)\cap U_k\neq \emptyset.$$
That is there is an element $g_1\in g_0\Gamma_{k-1}g_0^{-1}$ such that $g_1g_0b_k\Gamma\in U_k$. Hence we have $$g_1g_0(b_1\Gamma,\ldots,b_{k-1}\Gamma,b_k\Gamma)\in U_1\times \cdots\times U_{k-1}\times U_k=U,$$which completes the induction.
\end{proof}

We continue the proof of Theorem \ref{t1}. Let $\pi_\ell$ be the map from $(G/\Gamma)^\ell$ to $\mathcal K(G/\Gamma)$, the space of subsets of $G/\Gamma$, defined by $\pi_\ell(x_1,\ldots,x_\ell)=\{x_1,\ldots,x_\ell\}$. Observe that for any $\epsilon>0$, as $\ell$ large enough, there is an open subset $U\subset (G/\Gamma)^\ell$ such that $\pi_\ell(\bar x)$ is $\epsilon$-dense for any $\bar x\in U$. Therefore applying Lemma \ref{k1}, Theorem \ref{t1} follows.

\section{Proof of Theorem \ref{t2}}

The proof is similar to that of Theorem \ref{t1}. However, since the orbit closure is quite involved, the argument is much more complicated.

Let $A$ be an arbitrary infinite subset of $\mathbb T^n$. Without loss of generality, assume $A$ is countable, and denote $A=\{a_1,\cdots, a_i,\cdots\}_{i\in \mathbb N}$. For any $\ell\ge 1$, let $d_\ell$ be the dimension of the linear $\mathbb Q$-spanning space of $\{a_1,\cdots, a_\ell\}$. If for all $1\le i\le \ell$, $a_i\in\mathbb Q^n/\mathbb Z^n $, then $d_\ell=0$. Note that $d_\ell$ is increasing if $\ell$ increases. Therefore the limit $\lim_{\ell\to\infty}d_\ell$ exists (possibly $\infty$). Let $r=r(A)=\lim_{\ell\to\infty}d_\ell$, we have $r\in \mathbb N\cup\{0,\infty\}$. We split the proof in the following three cases.

\textbf{Case 1: }$r=\infty$. Then for any $\ell\ge 1$, one can pick a subset $\{b_1,\cdots, b_\ell\}$ from $A$, such that the points $b_1,\cdots, b_\ell$ are rationally independent. By Theorem \ref{pt2}, the $\Gamma$ orbit closure of $(b_1,\cdots, b_\ell)$ is $(\mathbb T^n)^\ell$. Therefore, for any $\epsilon>0$, one can choose $\ell$ large enough and the points $b_1,\cdots, b_\ell$ from $A$ such that, there is a $\gamma\in\Gamma$ with the property that the set $\gamma\{b_1,\cdots, b_\ell\}$ is $\epsilon$-dense. We are done.

\textbf{Case 2: }$r=0$. In this case $A\subset \mathbb Q^n/\mathbb Z^n$. We will need the following useful result.

\begin{lemma}\label{tl1}
For any $\ell>0$, any open subset $U\subset (\mathbb T^n)^\ell$, there exist a $g\in \Gamma$ and $(b_1,\ldots,b_\ell)\in (\mathbb T^n)^\ell$ with $(b_i)\in A$, such that $g(b_1,\ldots,b_\ell)\in U$.
\end{lemma}
\begin{proof}
The proof is similar to that of Lemma \ref{k1}. It suffices to prove the case when $U=U_1\times \cdots\times U_\ell$, where $U_i\subset \mathbb T^n$ is an open subset. We prove this by induction on $\ell$.

When $\ell=1$, since the orbit of any point in $A$ is finite and $Card(A)=\infty$, by Theorem \ref{bq2}, there is a $\Gamma$ orbit that intersects the fixed $U_1$. Therefore, one can pick this point and find an element of $\Gamma$, satisfying the lemma.

Assume that when $\ell=k-1\ge 1$, the lemma is true. Now we prove it for $\ell=k$. Let $U=U_1\times \cdots\times U_k$ be an arbitrary open set. Apply the case $\ell=k-1$ for the first $k-1$ product $U_1\times \cdots\times U_{k-1}$, we thus obtain $(b_1,\ldots,b_{k-1})\in (\mathbb T^n)^{k-1}$ with $b_i\in A$, and $g_0(b_1,\ldots,b_{k-1})\in U_1\times \cdots\times U_{k-1}$ for some $g_0\in \Gamma$. Notice that the $\Gamma$ orbit closure of $(b_1,\ldots,b_{k-1})$ is finite, then the stabilizer of the point $g_0(b_1,\ldots,b_{k-1})$ is a discrete group $g_0\Gamma_{k-1}g_0^{-1}$, where $\Gamma_{k-1}$ is a finite index subgroup of $\Gamma$. Here $g_0\Gamma_{k-1}g_0^{-1}$ is still a subgroup of $\Gamma$.

Let $\mathcal A_k:=A\backslash\{b_i:1\le i\le k-1\}$, then $Card(\mathcal A_k)=\infty$. Since $U_k$ is an open set, by Theorem \ref{bq2}, it follows that there is a $b_k\in\mathcal A_k$ such that $$\left(g_0\Gamma_{k-1}g_0^{-1}(g_0b_k)\right)\cap U_k\neq \emptyset.$$
That is there is an element $g_1\in g_0\Gamma_{k-1}g_0^{-1}$ such that $g_1g_0b_k\in U_k$. Hence we have $$g_1g_0\in\Gamma,\;\text{and }g_1g_0(b_1,\ldots,b_{k-1},b_k)\in U_1\times \cdots\times U_{k-1}\times U_k=U,$$which completes the induction.
\end{proof}

Let $\pi_\ell$ be the map from $(\mathbb T^n)^\ell$ to $\mathcal K(\mathbb T^n)$, the space of subsets of $\mathbb T^n$, defined by $\pi_\ell(z_1,\ldots,z_\ell)=\{z_1,\ldots,z_\ell\}$. For any $\epsilon>0$, let $\ell$ be large enough, then there exists an open subset $U\subset (\mathbb T^n)^{\ell}$, such that the subset $\pi_\ell(z_1,\ldots,z_\ell)$ is $\epsilon$-dense for any $(z_1,\ldots,z_\ell)\in U$. By applying Lemma \ref{tl1} with the $\ell$ and $U$, we are done.

\textbf{Case 3: }$1\le r<\infty$. One can pick a subset $\{z_1,\cdots,z_r\}$ of $r$ elements from $A$ such that $z_1,\cdots,z_r$ are rationally independent and any other point in $A$ is a $\mathbb Q$ combination of $z_1,\cdots,z_r$ and $\mathbb Q^n/\mathbb Z^n$. Without loss of generality, assume that \textcircled{1} $\{z_1,\cdots,z_r\}=\{a_1,\cdots, a_r\}$, and let $\mathcal A_r=A\backslash\{a_1,\cdots,a_r\}$. Denote $\mathbf a=(a_1,\cdots, a_r)$, then we can rewrite $a_i$ as $q_i^0+\langle\mathbf q_i, \mathbf a\rangle:=q_i^0+\sum_{j=1}^rq_i^ja_j$, where $q_i^0\in \mathbb Q^n/\mathbb Z^n$ and $\mathbf q_i=(q_i^1,\cdots,q_i^r)\in \mathbb Q^r$.

If $\mathcal A_r\cap \mathbb Q^n/\mathbb Z^n$ is infinite, then we can play the game as in {\bf Case 2} and obtain the proof. On the other hand, if $\mathcal A_r\cap \mathbb Q^n/\mathbb Z^n$ is finite, we may remove the finitely many rational points which will not affect our result. Therefore, we assume afterwards that \textcircled{2} $\mathcal A_r\cap \mathbb Q^n/\mathbb Z^n=\emptyset$. We assume also that \textcircled{3} $\{\mathbf q_i\}_{i\in\mathbb N}$ does not intersect any $\mathbb Q$-hyperplane $q_0+\mathbb Q^{r-1}$ ($q_0\in \mathbb Q^n$) with infinitely many points. Otherwise, we may get a case of $r-1$, from where we can start over again.

\begin{lemma}\label{tl2}
For any positive integer $\ell$, and $(b_1,\ldots,b_\ell)\in (\mathbb T^n)^\ell$ with $b_j\in \mathcal A_r$ for $1\le j\le \ell$, then the $\Gamma$ orbit closure of $(a_1,\ldots,a_r,b_1,\ldots,b_\ell)$ in $(\mathbb T^n)^{r+\ell}$ is a finite union of affine manifolds, and each one of the affine manifolds is the image of an affine map from $(\mathbb T^n)^r$ to $(\mathbb T^n)^{r+\ell}$. In particular, the dimension of the affine manifold is $nr$.
\end{lemma}
\begin{proof}
This follows from Theorem \ref{bq2-0}, Theorem \ref{pt2} and the assumption on $a_1,\ldots,a_r$ and $\mathcal A_r$.
\end{proof}

We now describe the affine map appeared above. Consider the point $a_k=q_k^0+\langle\mathbf q_k, \mathbf a\rangle$, let $q_k^j=\frac{s_k^j}{t_k^j}$, with $s_k^j,t_k^j\in\mathbb Z$ and $(s_k^j,t_k^j)=1$, $t_k^j\ge 1$. If $q_k^j=0$, then set $s_k^j=0$, $t_k^j=1$. Then the affine map $\phi_h: (\mathbb T^n)^r\to (\mathbb T^n)^r\times \mathbb T^n$ is defined by $$\phi_h(x_1,\ldots, x_r)=(t_k^1x_1,\ldots,t_k^rx_r,h+\sum_{j=1}^rs_k^jx_j),$$ where $h\in \{\Gamma.q_k^0\}$. The corresponding orbit closure of $(a_1,\ldots,a_r,a_k)$ is given by $$\bigcup_{h\in \{\Gamma.q_k^0\}}\phi_h((\mathbb T^n)^r).$$Next, if there is another point $a_l=q_l^0+\langle\mathbf q_l, \mathbf a\rangle$ with $q_l^j=\frac{s_l^j}{t_l^j}$. Then the affine map is defined by $$\phi_h(x_1,\ldots, x_r)=(t_1x_1,\ldots,t_rx_r,h+(\sum_{j=1}^r\bar s_k^jx_j,\sum_{j=1}^r\bar s_l^jx_j)),$$where $h\in \{\Gamma.(q_k^0,q_l^0)\}\subset(\mathbb T^n)^2$, $t_j=\frac{t^j_kt^j_l}{(t^j_k,t^j_l)}$, $\bar s^j_k=\frac{s_k^jt^j_l}{(t^j_k,t^j_l)}$ and $\bar s^j_l=\frac{s_l^jt^j_k}{(t^j_k,t^j_l)}$. The corresponding orbit closure of $(a_1,\ldots,a_r,a_k,a_l)$ is given by $$\bigcup_{h\in \{\Gamma.(q_k^0,q_l^0)\}}\phi_h((\mathbb T^n)^r).$$One can define similarly for the case when $\ell\ge 3$, which is even more complicated. We choose not to do the cumbersome work here but hope the construction is clear enough.

\begin{lemma}\label{tl3}
If $B\subset \mathcal A_r$ is an infinite subset, then for any open subset $V\subset (\mathbb T^n)^r$ and any open subset $U\subset\mathbb T^n$, there exist two points $\hat b,\bar b\in B$ such that 
\begin{itemize}
\item the orbit closure of $(a_1,\ldots,a_r,\hat b,\bar b)$ has non empty intersection with $(\mathbb T^n)^r\times U\times \mathbb T^n$;
\item the preimage of the intersection under the affine map has non empty intersection with $V$.\end{itemize}
\end{lemma}
\begin{proof}
Since by assumption that $\mathcal A_r\cap \mathbb Q^n/\mathbb Z^n=\emptyset$, $B$ contains only irrational points. Pick any one of them, say $a_k=q_k^0+\langle\mathbf q_k, \mathbf a\rangle\notin\mathbb Q^n/\mathbb Z^n$. Then by Lemma \ref{tl2}, the orbit closure of $(a_1,\ldots,a_r,a_k)$ is a graph defined by some affine map $\phi:(\mathbb T^n)^r\to (\mathbb T^n)^r\times \mathbb T^n$, and must have nontrivial intersection with $(\mathbb T^n)^r\times U$ since the $\Gamma$ orbit closure of $a_k$ is $\mathbb T^n$. This intersection is open in the orbit closure because $U$ is open.

Now have the construction of affine maps in mind, the second assertion is equivalent to: for some $a_k$, there is a $a_l$ such that $$\exists \;(x_1,\ldots,x_r)\in V\text{, such that }\bigcup_{h\in \{\Gamma.(q_k^0,q_l^0)\}}\{h+(\sum_{j=1}^r\bar s_k^jx_j,\sum_{j=1}^r\bar s_l^jx_j)\}\subset U\times \mathbb T^n.$$As $V$ is open, this is true when $\max_j\{|\bar s_k^j|\}$ is large enough. By assumption \textcircled{3}, since $B$ is an infinite subset, we can choose an $a_l$ so that some $t_l^j$ is large enough (so $|\bar s^j_k|=\frac{|s_k^j|t^j_l}{(t^j_k,t^j_l)}$ is large enough). The proof is complete by making $\hat b=a_k$ and $\bar b=a_l$.
\end{proof}

\begin{lemma}\label{k4}
For any positive integer $\ell$, and any open subset $U\subset (\mathbb T^n)^{\ell}$, there exist $(b_1,\ldots,b_\ell)\in (\mathbb T^n)^\ell$ and $(c_1,\ldots,c_\ell)\in (\mathbb T^n)^\ell$ with $b_j,c_j\in \mathcal A_r$ for $1\le j\le \ell$, such that 
\begin{itemize}
\item the $\Gamma$ orbit closure of $(a_1,\ldots,a_r,b_1,\ldots,b_\ell,c_1,\ldots,c_\ell)$ in $(\mathbb T^n)^{r+2\ell}$ has non empty intersection with $(\mathbb T^n)^r\times U\times(\mathbb T^n)^\ell$;
\item the intersection is open when restricted in the orbit closure (affine submanifold).
\end{itemize}
\end{lemma}
\begin{proof}
Firstly, note that by the assumption on $a_1,\ldots,a_r$, the $\Gamma$ orbit closure of $(a_1,\cdots, a_r)$ is $(\mathbb T^n)^r$. Next, as in the previous two lemmas, it suffices to prove the case when $U=U_1\times \cdots\times U_\ell$, where $U_i\subset \mathbb T^n$ is an open subset. We prove this by induction on $\ell$.

When $\ell=1$, this is the content of Lemma \ref{tl3}.

Assume that when $\ell=k-1\ge 1$, the lemma is true. Now we prove it for $\ell=k$. Let $U=U_1\times \cdots\times U_k$ be an arbitrary open set. Apply the case $\ell=k-1$ for the first $k-1$ product $U_1\times \cdots\times U_{k-1}$, and let $W$ be the intersection resulted. By Lemma \ref{tl2}, $W$ is the intersection of the image of an affine map with $(\mathbb T^n)^r\times U_1\times \cdots\times U_{k-1}\times(\mathbb T^n)^{k-1}$. Let $V\subset (\mathbb T^n)^r$ be the preimage. Since $W$ is open in the orbit closure, it follows that $V$ is an open set of $(\mathbb T^n)^r$. Now apply Lemma \ref{tl3} for $V$, $U_k$ and $B=\mathcal A_r\backslash \{b_1,\ldots, b_{k-1},c_1,\ldots,c_{k-1}\}$, we have two points $\hat b$ and $\bar b$ satisfying that 
\begin{itemize}
\item the orbit closure of $(a_1,\ldots,a_r,\hat b,\bar b)$ has non empty intersection with $(\mathbb T^n)^r\times U_k\times \mathbb T^n$;
\item the preimage of the intersection under the affine map has non empty intersection with $V$.
\end{itemize}
Let $b_k=\hat b$ and $c_k=\bar b$, then $(b_1,\ldots,b_k)$ and $(c_1,\ldots,c_k)$ satisfies the lemma. Hence the induction is complete and the proof is done.
\end{proof}

Continue the proof of {\bf Case 3}. Let $\pi_\ell$ be the map from $(\mathbb T^n)^\ell$ to $\mathcal K(\mathbb T^n)$, the space of subsets of $\mathbb T^n$, defined by $\pi_\ell(z_1,\ldots,z_\ell)=\{z_1,\ldots,z_\ell\}$. For any $\epsilon>0$, let $\ell$ be large enough, then there exists an open subset $U\subset (\mathbb T^n)^{\ell}$, such that the subset $\pi_\ell(z_1,\ldots,z_\ell)$ is $\epsilon$-dense for any $(z_1,\ldots,z_\ell)\in U$. Apply Lemma \ref{k4} with the $\ell$ and $U$, there exists a $\gamma\in \Gamma$ such that $\gamma A$ is $\epsilon$-dense. The proof is complete.


\text{\quad}\\

\textsc{Department of Mathematics, The Pennsylvania State University, University Park, PA 16802, USA}

{Email: cud159@psu.edu}


\begin{thebibliography}{99}


\bibitem{ap} N. Alon and Y. Peres, Uniform dilations, Geom. Funct. Anal. {\bf2} (1992), 1--28.

\bibitem{bq} Y. Benoist and Jean-François Quint, Stationary measures and invariant subsets of homogeneous spaces (III), Annals of Mathematics {\bf178}.3 (2013): 1017--1059.

\bibitem{bp} D. Berend and Y. Peres, Asymptotically dense dilations of sets on the circle, J. London Math. Soc. (2) {\bf47} (1993), 1--17.

\bibitem{do1} C. Dong. On density of infinite subset I. preprint.

\bibitem{gl} S. Glasner, Almost periodic sets and measures on the torus, Israel J. Math. {\bf32} (1979), 161--172.


\bibitem{kl} M. Kelly and T. L\^e, Uniform dilations in higher dimensions, J. London Math. Soc. (2) {\bf88} (2013), 925--940.

\bibitem{ma} G. A. Margulis, Discrete subgroups of semisimple Lie groups, Ergebnisse der Mathematik und ihrer Grenzgebiete (3) [Results in Mathematics and Related Areas (3)], vol. {\bf17}, Springer-Verlag, Berlin, 1991.

\bibitem{mo} S. Mozes. Epimorphic subgroups and invariant measures. Ergod. Th. \& Dynam. Sys. {\bf15} (1995), 1207--1210.

\bibitem{sw} N. Shah and Barak Weiss. On actions of epimorphic subgroups on homogeneous spaces. Ergodic Theory and Dynamical Systems {\bf20}.02 (2000): 567--592.

\bibitem{zi} Robert J. Zimmer, Ergodic theory and semisimple groups, Monographs in Mathematics, vol. {\bf81}, Birkh\"auser Verlag, Basel, 1984.


\end{thebibliography}
\end{document}